\RequirePackage{fix-cm}
\documentclass[smallcondensed]{svjour3}     
\smartqed  
 \usepackage{amssymb}
\usepackage{graphicx}
\usepackage{amsmath}
\usepackage[english]{babel}
\usepackage[latin1]{inputenc}
\usepackage{textcomp}
\usepackage{longtable}
\usepackage{listings}
\usepackage{enumerate}
\usepackage{color}
\begin{document}

\title{Bounding robustness in complex networks under topological changes through majorization techniques
}


\author{G.P.Clemente         \and
        A.Cornaro }


\institute{G.P. Clemente \at
              Department of Mathematical Sciences, Mathematical Finance and Econometrics \\ Università  Cattolica del Sacro Cuore \\
              \email{gianpaolo.clemente@unicatt.it}           
           \and
           A. Cornaro \at
            Department of Mathematical Sciences, Mathematical Finance and Econometrics \\ 
              Università  Cattolica del Sacro Cuore \\
               \email{alessandra.cornaro@unicatt.it}           
}

\date{}

\maketitle

\fontsize{9.15}{11.5}\selectfont
\begin{abstract}
Measuring robustness is a fundamental task for analyzing the structure of complex networks. Indeed, several approaches to capture the robustness properties of a network have been proposed.	
In this paper we focus on spectral graph theory where robustness is measured by means of a graph invariant called Kirchhoff index, expressed in terms of eigenvalues of the Laplacian matrix associated to a graph. This graph metric is highly informative as a robustness indicator for several real-world networks that can be modeled as graphs.
We discuss a methodology aimed at obtaining some new and tighter
bounds of this graph invariant when links are added or removed. We take advantage of real analysis techniques, based on majorization theory and optimization of functions which preserve the majorization order (Schur-convex functions).
Applications to simulated graphs show the effectiveness of our bounds, also in providing meaningful insights with respect to the results obtained in the literature.
\end{abstract}
\section{Introduction}
\label{intro}

Assessing and improving robustness of complex networks is a challenge that has gained increasing attention in the literature. Network robustness research has indeed been carried out by scientists with different backgrounds, like mathematics, physics, computer science and biology. As a result, quite a lot of different approaches to capture the robustness properties of a network have been undertaken. Traditionally, the concept of robustness was mainly centered on graph connectivity. Recently, a more contemporary definition has been developed. According to \cite{sydney}, it is defined as the ability of a network to maintain its total throughput under node and link removal. Under this definition, the dynamic processes that run over a network must be taken into consideration. \\
In this framework several robustness metrics based on network topology or spectral graph theory have been proposed (see \cite{Boccaletti},\cite{Costa},\cite{Dorogo}). In particular, we focus on spectral graph theory where robustness is measured by means of functions of eigenvalues of the Laplacian matrix associated to a graph (\cite{Ellens},\cite{VanMieghem}). Indeed, this paper is aimed to the inspection of a graph measure called \emph{effective graph resistance}, also known as \emph{Kirchhoff index} (or \emph{resistance distance}), derived from the field of electric circuit analysis (\cite{Klein}). The Kirchhoff index has undergone intense scrutiny in
recent years and a variety of techniques have been used, including
graph theory, algebra (the study of the Laplacian and of the normalized Laplacian),
electric networks, probabilistic arguments involving hitting times of random
walks (\cite{Broder},\cite{Chandra}) and discrete potential theory (equilibrium measures and Wiener capacities),
among others. It is defined as the accumulated effective resistance between all pairs of vertices.\\ 
It is worth pointing out that the Kirchhoff index can be highly valuable and informative as a robustness measure of a network, showing the ability of a network to continue performing well when it is subject to failure and/or attack (see, for instance, \cite{WangP}). In fact, the pairwise effective resistance measures the vulnerability of a connection between a pair of vertices that considers both the number of paths between the vertices and their length. Therefore, a small value of the effective graph resistance indicates a robust network. In this framework, several works studied the Kirchhoff index in networks under topological changes, such as characterized by a link addition or removal. For example, Ghosh et al \cite{Gosh} study the minimization of the effective graph resistance by allocating link weights in weighted graphs.
Van Mieghem et al in \cite{VanM2} show the relation between the Kirchhoff index and the linear degree correlation coefficient. Abbas et al in \cite{Abbas} reduce the Kirchhoff index of a graph by adding links in a step-wise way. Finally, Wang et al. in \cite{WangP} focus on Kirchhoff index as an indicator of robustness in complex networks when a single link is added or removed. In particular, being the calculation of this index computationally intensive for large networks, they provide upper and lower bounds after one link addition or removal (see \cite{dehmer}, \cite{BCT2}). \\
In this paper, we discuss a methodology aimed at obtaining some new and tighter bounds of this graph invariant when links are added or removed.  Our strategy takes advantage of real analysis techniques, based on majorization theory and optimization of functions which preserve the majorization order, the so-called \emph{Schur-convex} functions.  One major advantage of this approach is
to provide a unified framework for recovering many well-known upper and lower bounds obtained with a variety of methods, as well
as providing  better ones. It is worth pointing out that the localization of topological indices is typically carried out by applying
classical inequalities such as the Cauchy-Schwarz inequality or the arithmetic-geometric-harmonic mean inequalities.
Within this framework, we provide new lower bounds when one or more links are added or removed. This proposal represents a novelty because, to the best of our knowledge, the existing bounds (see \cite{WangP}) are based on the assessment of robustness after only one link addition or removal. Additionally, even in this case, we show that our bounds perform better.\\
The remainder of this paper is organized as follows. In Section \ref{not} some notations and preliminaries are given. In Section \ref{bou} we provide new lower bounds for the Kirchhoff Index under topological changes. Section \ref{num} shows how the bounds determined in Section \ref{bou} improve those presented in the literature. Conclusions follow.

\section{Notation and Preliminaries}\label{not}
\subsection{Basic graph concept and the Kirchhoff index}
 Let us firstly recall some basic graph concepts.

\noindent Let $G=(V,E)$ be a simple, connected and undirected graph, where $%
V=\left\{ v_{1},...,v_{n}\right\} $ is the set of vertices and $E\subseteq
V\times V$ is the set of links. We consider graphs with fixed order $\left\vert
V\right\vert =n\ $ and fixed size $\left\vert E\right\vert =m.$ Let $\mathbf{\pi}
=(d_{1},d_{2},..,d_{n})$ denote the degree sequence of $G$ arranged in non-increasing order $d_{1}\geq d_{2}\geq \cdots \geq d_{n}$, being $d_{i}$ the
degree of vertex $v_{i}$. It is well known that $\overset{n}{\underset{i=1}{\sum }}d_{i}=2m$.  \\
Let $A(G)$ be the adjacency matrix of $G$ and $\lambda
_{1}\geq \lambda _{2}\geq ...\geq \lambda _{n}$ be its (real)
eigenvalues. Given the diagonal matrix $D(G)$ of vertex degrees, the matrix $%
L(G)=D(G)-A(G)$ is known as the Laplacian matrix of $G$. Let $\mu _{1}\geq
\mu _{2}\geq ...\geq \mu _{n}$ be its eigenvalues. Hence, we can express the
following well-known properties of spectra of $L(G)$:%
\begin{equation*}
\overset{n}{\underset{i=1}{\sum }}\mu _{i}=2m;\,\,\,\mu _{1}\geq
1+d_{1}\geq \dfrac{2m}{n};\,\,\,\mu_{2}\geq d_{2};\,\,\,\mu _{n}=0,\text{ }\mu _{n-1}>0.
\end{equation*}%
For more details refer to \cite{Grone1994} and \cite{Harary}.

The \textit{Kirchhoff index} $K(G)$ of a simple connected graph $G$ was defined by Klein
and Randi\'{c} in \cite{Klein} as \[K(G)=\sum_{i<j}R_{ij},\]
where $R_{ij}$ is the effective resistance between vertices $i$ and $j$,
which can be computed using Ohm's law. \\
In addition to its original definition, the Kirchhoff index can be rewritten as
\begin{equation}
	K(G)=n\sum_{i=1}^{n-1}{\frac{1}{{\mu_{i}}}},  \label{kirchhoff}
\end{equation}%
in terms of the eigenvalues of the Laplacian matrix $L$ (see \cite{Gutman}, \cite{Zhu}).

\subsection{Majorization theory}

We briefly recall some basic concepts about majorization order and Schur
convexity. For more details see \cite{dehmer,BCT1,Marshall}.

\begin{definition}
	Given two vectors $\mathbf{y}$, $\mathbf{z\in }$ $D=\{\mathbf{x}\in \mathbb{R%
	}^{n}:x_{1}\geq x_{2}\geq ...\geq x_{n}\}$, the majorization order $\mathbf{y%
	}\trianglelefteq \mathbf{z}$ means:
	\begin{equation*}
	\begin{cases}
	\left\langle \mathbf{y},\mathbf{s}^{\mathbf{k}}\right\rangle \leq
	\left\langle \mathbf{z},\mathbf{s}^{\mathbf{k}}\right\rangle ,\text{ }%
	k=1,...,(n-1) \\[3mm]
	\left\langle \mathbf{y},\mathbf{s}^{\mathbf{n}}\right\rangle =\left\langle
	\mathbf{z},\mathbf{s}^{\mathbf{n}}\right\rangle%
	\end{cases}%
	\end{equation*}%
	where $\left\langle \cdot ,\cdot \right\rangle $ is the inner product in $%
	\mathbb{R}^{n}$ and $\mathbf{s^{j}}=[\underbrace{1,1,\cdots ,1}_{j},%
	\underbrace{0,0,\cdots 0}_{n-j}],\quad j=1,2,\cdots ,n.$
\end{definition}

\noindent Given a closed subset $S\subseteq \Sigma_{a} = D\cap \{\mathbf{x}%
\in \mathbb{R}_{+}^{n}:\left\langle \mathbf{x}, \mathbf{s^{n}}\right\rangle
=a\}$, where $a$ is a positive real number, let us consider the following
optimization problem
\begin{equation}  \label{eq:minimization}
\mathrm{Min}_{\mathbf{x} \in S} \,\, \phi (\mathbf{x}).
\end{equation}

\noindent If the objective function $\phi$ is Schur-convex, i.e. $\mathbf{x}%
\trianglelefteq \mathbf{y}$ implies $\phi (\mathbf{x})\leq $ $\phi (\mathbf{y%
})$, and the set $S$ has a minimal element $\mathbf{x}_{\ast }(S)$ with
respect to the majorization order, then $\mathbf{x}_{\ast }(S)$ solves
problem \eqref{eq:minimization},  that is
\begin{equation*}
\phi(\mathbf{x}) \ge \phi(\mathbf{x} _{\ast }(S)) \,\, \mathrm{\ for } \,
\mathrm{all} \,\, \mathbf{x} \in S.
\end{equation*}

\noindent It is worthwhile to notice that if $S^{\prime }\subseteq S$ the inequality $\mathbf{x}_{\ast }(S)\trianglelefteq \mathbf{x}
_{\ast }(S^{\prime })$ holds and thus
\begin{equation}
\phi (\mathbf{x})\geq \phi (\mathbf{x}_{\ast }(S^{\prime }))\geq \phi (%
\mathbf{x}_{\ast }(S))\,\,\mathrm{for}\,\mathrm{all}\,\,\mathbf{x}\in
S^{\prime }.  \label{1.a}
\end{equation}

On the other hand, if the objective function $\phi $ is Schur-concave, i.e. $%
-\phi $ is Schur-convex, then
\begin{equation}
\phi (\mathbf{x})\leq \phi (\mathbf{x}_{\ast }(S^{\prime }))\leq \phi (%
\mathbf{x}_{\ast }(S)\,\,\mathrm{for}\,\mathrm{all}\,\,\mathbf{x}\in
S^{\prime }.  \label{1.b}
\end{equation}

A very important class of Schur-convex (Schur-concave) functions can be
built adding convex (concave) functions of one variable. Indeed, given an
interval $I\subset
\mathbb{R}
$, and a convex function $g:I\rightarrow
\mathbb{R}
$, the function $\phi (\mathbf{x})=\sum_{i=1}^{n}g(x_{i})$ is Schur-convex
on $I^{n}=\underbrace{I\times I \times \cdots \times I}_{n- times}$. The
corresponding result holds if $g$ is concave on $I^n$.

\noindent In \cite{BCT1}, the authors derived the maximal and minimal
elements, with respect to the majorization order, of the set
\begin{equation*}
S_{a}=\Sigma _{a}\cap \{\mathbf{x}\in \mathbb{R}^{n}:M_{i}\geq x_{i}\geq
m_{i},i=1,\cdots ,n\}
\end{equation*}
where $M_{1}\geq M_{2}\geq \cdots \geq M_{n},$ $m_{1}\geq m_{2},\cdots \geq
m_{n}$.

In particular, in the sequel, we need the following result.

\begin{theorem}
	(see \cite{BCT1}, Theorem 8)
	\label{th:minimo} Let $k \ge 0$ and $d \ge 0$ be the smallest integers such
	that
	
	\begin{itemize}
		\item[1)] $k + d <n $
		
		\item[2)] $m_{k+1} \le \rho \le M_{n-d}$ where $\rho= \dfrac { a - \sum_{i=1}^k m_{i}-\sum_{i=n-d+1}^n M_{i}}{n-k-d}$.
	\end{itemize}
	
	Then
	\begin{equation*}
	\mathbf{x_{\ast }}(S_{a})=[m_1, \cdots , m_k, \rho^{n-d-k},
	M_{n-d+1} \cdots, M_n].
	\end{equation*}
\end{theorem}

From Theorem \ref{th:minimo} the next corollaries follow

\begin{corollary}
	(see \cite{Marshall})\label{cor:marshall2} Let $0\leq m<M$ and $m\leq \dfrac{%
		a}{n}\leq M.$ Then $x_{\ast }(S_{1})=\dfrac{a}{n}\mathbf{s^{n}}$.
\end{corollary}

\begin{corollary}
	(see \cite{BCT1}, Corollary 14) \label{theorem:S_2} Let us consider the set
	\begin{equation*}
	S_{2}^{\left[ h\right] }=\Sigma _{a}\cap \{\mathbf{x}\in \mathbb{R}%
	^{n}:x_{i}\geq \alpha ,\text{ }i=1,\cdots ,h,\text{ }1\leq h\leq n,\text{ }%
	0<\alpha \leq \frac{a}{h}\}  \label{S2}
	\end{equation*}%
	Then
	\begin{equation*}
	x_{\ast }(S_{2}^{\left[ h\right] })=\left\{
	\begin{array}{cc}
	\dfrac{a}{n}\mathbf{s^{n}} & \text{ if }\alpha \leq \dfrac{a}{n} \\
	\alpha \mathbf{s^{h}}+\rho \mathbf{v^{h}}\text{ with }\rho =\dfrac{a-\alpha h%
	}{n-h} & \text{ if }\alpha >\dfrac{a}{n}%
	\end{array}%
	\right. .
	\end{equation*}
\end{corollary}


\section{Main results}\label{bou}

A huge and extended literature focuses on the localization of the Kirchhoff Index. A variety of general bounds for $K(G)$ have been found in terms of invariants of $G$, such as $|V|$, $|E|$, etc. (see, for instance, \cite{BCPT1},\cite{Pal2011},\cite{Wang} and \cite{ZhouTrina1}). \\
Less attention has been paid in measuring the highest and the lowest values that this graph metric can assume after certain changes in the network structure. In this regard, bounds can be useful for robustness assessment under topological changes, such as links addition or removal. 
In this context, Wang et al. in \cite{WangP} provided the following lower bounds:
\begin{equation}
K(G+e)\geq \dfrac{K(G)}{1+ \dfrac{n \rho}{2}},
\label{KG2}
\end{equation}
where $K(G+e)$ is the Kirchhoff index computed on the graph $G+e$ resulting after one link addition to $G$ and $\rho$ is the diameter of $G$, and:

\begin{equation}
K(G-e)\geq\dfrac{n(n-1)^2}{
	2(m-1)},
\label{KG1}
\end{equation}
where $K(G-e)$ is the Kirchhoff index computed on the graph $G-e$ resulting after one link removal to $G$.

In this paper, we aim at obtaining new lower bounds after that one or more links are added or removed. 
In case of \emph{links addition} we provide the following theorem:

\begin{theorem} 
	For any simple graph  $G'(n,m+h)$, obtained by the simple and connected graph $G(n,m)$ by adding $1\leq h \leq \dfrac{n(n-1)}{2}-m $ links\footnote{Notice that when $h = \dfrac{n(n-1)}{2}-m$ the complete graph is attained.}, we have:
	\begin{equation}
	K(G') \geq n\left(\frac{1}{d_{1}+1}+\frac{1}{d_{2}}+\frac{(n-3)^2}{2m+2h-1-d_{1}-d_{2}}\right),
	\label{KG'}
	\end{equation}
	where $d_{1}$ and $d_{2}$ are the first and the second largest degrees of vertices of $G$.
\end{theorem}

\begin{proof} 
We first prove that the localization of the first and the second eigenvalues of the new graph obtained under links addition depends on the first and second largest degrees of the original graph. After one link addition ($h=1$), graph $G(n,m)$ becomes graph $G'(n,m+1)$, whose eigenvalues can be defined as $\mu_{j}(G')=\mu_{j}+\Delta \mu_{j}$. Hence, $\overset{n-1}{\underset{j=1}{\sum }}\left(\mu _{j}+\Delta \mu_{j}\right)=2m+2$ holds. \\
The increase $\Delta\mu _{j}$ of the eigenvalue satisfies 
\begin{equation}
\overset{n-1}{\underset{j=1}{\sum }}\Delta\mu _{j}=2m+2-\overset{n-1}{\underset{j=1}{\sum }}\mu _{j}=2m+2-2m=2.
\label{sumdeltamu}
\end{equation}
By the interlacing property (see \cite{VanMieghem}), we have 
\begin{equation}
\mu _{j}\leq \mu _{j}+\Delta \mu_{j}\leq \mu _{j-1},
\label{eq:int}
\end{equation}
showing that $0 \leq \Delta \mu_{j} \leq 2$ for $j=2,...,n-1$. \\
From (\ref{eq:int}), it easily follows that $\mu _{2}(G)\leq \mu _{2}(G')\leq \mu _{1}(G)\leq \mu _{1}(G')$ and by making use of the well-known inequalities $\mu _{1}(G)\geq 1+d_{1}(G)$ and $\mu _{2}(G)\geq d_{2}(G)$, we have that:
\begin{equation*}
\mu _{1}(G')\geq 1+d_{1}(G); \,\,\,  
\mu _{2}(G')\geq d_{2}(G).
\end{equation*}

Extending the previous reasoning to $G'(n,m+h)$, with  $1\leq h \leq \dfrac{n(n-1)}{2}-m $,  we now obtain $\overset{n-1}{\underset{j=1}{\sum }}\left(\mu _{j}+\Delta \mu_{j}\right)=2m+2h$ and $\overset{n-1}{\underset{j=1}{\sum }}\Delta\mu _{j}=2h$. By the interlacing property, we have
 $0 \leq \Delta \mu_{j} \leq 2h$ for $j=2,...,n-1$. As before, in a iterative way, we can show that
\begin{equation}
\mu _{1}(G')\geq 1+d_{1}(G); \,\,\,  \mu _{2}(G')\geq d_{2}(G). 
\label{ineqh1}
\end{equation}
In what follows $d_{1}(G)=d_{1}$ and $d_{2}(G) =d_{2}.$ \\
The inequalities in (\ref{ineqh1}) can be encoded in the following set:
\begin{equation*}
S^{G'}=\Sigma _{2m+2h}\cap \{\mathbf{\mu }\in \mathbb{R}^{n-1}:\mu _{1}(G')\geq
1+d_{1},\,\,\mu _{2}(G')\geq d_{2}\}.
\end{equation*}

Under the assumption $2m+2h\leq 1+d_{1}+(n-2)d_{2}$, by Corollary 2 the minimal element is given by:
\begin{equation*}
\mathbf{x}_{\ast }(S^{G'})=\left[ 1+d_{1},d_{2},\underbrace{\frac{2m+2h-1-d_{1}-d_{2}%
	}{n-3},\cdots ,\frac{2m+2h-1-d_{1}-d_{2}}{n-3}}_{n-3}\right],
\end{equation*}
and, by the fact that the Kirchhoff index, as defined in formula (\ref{kirchhoff}), is a Schur-convex funcion of its arguments, we get:
\begin{equation}
K(G') \geq n\left(\frac{1}{d_{1}+1}+\frac{1}{d_{2}}+\frac{(n-3)^2}{2m+2h-1-d_{1}-d_{2}}\right).
\end{equation}

This lower bounds displays strict monotonicity when edges are added and this is a desirable property for a robustness quantifier. Furthermore, it depends only on the features of graph $G$ (i.e., $m$, $n$, $d_{1}$ and $d_{2}$).

\end{proof}

\vspace{10pt}
When we consider the case of \emph{links removal}, we provide the following theorem:
\begin{theorem} 
	For any simple connected graph  $G''(n,m-h)$, obtained by the simple and connected graph $G(n,m)$ by removing $1\leq h \leq n-1-m $ links, we have:
	\begin{equation}
K(G'') \geq n\left(\frac{1}{d_{1}+1-2h}+\frac{1}{d_{2}-2h}+\frac{(n-3)^2}{2m+2h-1-d_{1}-d_{2}}\right),
	\label{KG''}
	\end{equation}
	where $d_{1}$ and $d_{2}$ are the first and the second largest degrees of vertices of $G$ and $h<\frac{d_{2}}{2}$.
\end{theorem} 

\begin{proof} 
	Let us consider $G''(n,m-h)$ with $1\leq h \leq n-1-m $. \\
	Let $\mu_{j}(G'')=\mu_{j}-\Delta \mu_{j}$. Since $\overset{n-1}{\underset{j=1}{\sum }}\left(\mu _{j}-\Delta \mu_{j}\right)=2m-2h$ we have that $\overset{n-1}{\underset{j=1}{\sum }}\Delta\mu _{j}=2h$. 
Hence, $0 \leq \mu_{j}(G)-\mu _{j}(G'') \leq 2h$ with $j=1,...,n-1$,

\begin{equation}
\mu _{1}(G'')\geq 1+d_{1}-2h \,\,\,  ,\mu _{2}(G'')\geq d_{2}-2h,
\label{ineqh2}
\end{equation}
with $h<\frac{d_{2}}{2}.$

We can make use of the inequalities in (\ref{ineqh2}) dealing with the set:
\begin{equation*}
S^{G''}=\Sigma _{2m-2h}\cap \{\mathbf{\mu }\in \mathbb{R}^{n-1}:\mu _{1}(G'')\geq
1+d_{1}-2h,\,\,\mu _{2}(G'')\geq d_{2}-2h\}
\end{equation*}%

and we get:
\begin{equation}
K(G'') \geq n\left(\frac{1}{d_{1}+1-2h}+\frac{1}{d_{2}-2h}+\frac{(n-3)^2}{2m+2h-1-d_{1}-d_{2}}\right),
\label{KG''}
\end{equation}
with $h<d_{2}/2.$

\end{proof}

It is noteworthy that we can also easily recover bound (\ref{KG1}) by using majorization order, applying the well-known results provided by Marshall and Olkin in \cite{Marshall}. Additionally, since the set $S^{G''}$ is a more specific closed set of constraints with respect to $S=\Sigma_{2m-2h}\cap \{\mathbf{\mu }\in \mathbb{R}^{n-1}: \mu _{1}\geq
\mu _{2}\geq ...\geq \mu _{n-1} \}$, by (\ref{1.a}) we have that bound (\ref{KG''}), for $h=1$, is always tighter than bound (\ref{KG1}).

\section{Empirical analysis}\label{num}
In this Section, we compare our bounds (\ref{KG'}) and 	(\ref{KG''}) with those existing in the literature. We initially compute them for graphs generated by using the Erd\'os-R\'enyi  (ER) model $G_{ER}(n; p)$ (see \cite{Boll},\cite{ER59},\cite{ER60}) where links are included with probability $p$ independent from every other link. Graphs have been derived randomly and by ensuring that are connected. Hence, given the random graph $G$, either one link, between two distinct and not adjacent vertices, is randomly added or one exiting link is removed, obtaining the graphs $G'$ and $G''$ respectively\footnote{It is noteworthy that, in order to obtain a connected graph $G''$, we discard simulations where the obtained graph is not connected.}. \\
To this end, Table \ref{tab:table1} compares alternative lower bounds of $K(G')$. In this case, we initially generate a random graph $G$
with the $G_{ER}(n; 0.5)$ model and considering different number of vertices. The value of the Kirchhoff Index and the density have been reported. As expected, we have that the ratio, between the actual number of links and the maximum possible number of links, moves around $p$. Hence, we have medium clustered networks. We proceed then comparing existing lower bounds of $K(G')$. Each graph $G'$ has been derived, as previously described, after one link addition. It is noteworthy that bound (\ref{KG'}) has the best performance, while bound (\ref{KG2}), that has been proposed in \cite{WangP}, leads to a very rough approximation of $K(G')$. This behaviour is mainly explained by the presence of the diameter $\rho$ at the denominator of formula (\ref{KG2}). Although the $ER$ model has a very low diameter\footnote{Notice that the diameter of $ER$ model tends to 2 for larger graphs}, the limitation provided in \cite{WangP} is very far from the exact value. Obviously for graphs with a large diameter, even worse results would be obtained. 

In case of graphs $G''$ (see Table \ref{tab:table2}), derived after one link removal, we observe that bound (\ref{KG''}) has the best performance. When large graphs are considered, the improvements are very slight with respect to bound (\ref{KG1}). \\
For the sake of brevity, we have only reported the results for the ER graphs, but we have numerically checked that our proposal also improves existing bounds on other classes of graphs (as Watts and Strogatz \cite{WS} or Barab\'asi-Albert model \cite{BA}).

\begin{table}[!h]
	\centering
	\caption{Comparison between lower bounds of $K(G)$ after one link addition. Graphs are randomly generated by using Erd\"os-R\'enyi (ER) model, varying the number of vertices $n$ and with probability $p=0.5$.}
	\begin{tabular}{|l | c| c | c |c | c| c| c|}
		\hline
		$n$	& $K(G)$	& $K(G')$	& Our bound (\ref{KG'}) & Bound Wang et al. (\ref{KG2}) & Density \\
		\hline
		10	&19.86	&18.97	&16.53	&1.81  & 0.55	 \\
		20	&48.71	&47.80	&43.57	&1.57 & 0.46   \\
		30	&66.21	&65.70	&60.80	&2.14	& 0.47 \\
		40	&78.64	&78.49	&75.63	&1.92  & 0.51	\\
		50	&106.76	&106.59	&102.06	&2.09  & 0.48	\\
		100	&191.36	&191.28	&188.38	&1.89 & 0.53 \\	
		200	&409.23	&409.19	&405.21	&2.04 & 0.49    \\
		500	&1001.28	&1001.26	&997.12	&2.00  & 0.50	\\
		1000	&1999.25	&1999.24	&1995.26	&2.00 & 0.50\\ \hline \hline
	\end{tabular}
\label{tab:table1}
\end{table}

\begin{table}[!h]
\centering
\caption{Comparison between lower bounds of $K(G)$ after one link removal. Graphs are randomly generated by using Erd\"os-R\'enyi  (ER) model, varying the number of vertices $n$ and with probability $p=0.5$.}
\begin{tabular}{|l | c| c | c |c |c| }
	\hline	
	$n$	& $K(G)$	& $K(G'')$	& Our bound (\ref{KG''}) & Bound Wang et al. (\ref{KG1}) & Density\\
	\hline
	10	&19.86	&21.17	& 17.78	& 17.61 & 0.55\\
	20	&48.71	&49.26	& 44.31	& 44.02 & 0.46\\
	30	&66.21	&66.69	& 61.16	& 60.94 & 0.47\\
	40	&78.64	&78.81	& 75.87	& 75.67 & 0.51\\
	50	&106.76	&106.92	& 102.28	& 102.08 & 0.48\\
	100	&191.36	&191.43	& 188.49	& 188.41 & 0.53\\
	200	&409.23	&409.27	& 405.28	& 405.21 & 0.49\\
	500	&1001.28	&1001.29	& 997.15	& 997.10 & 0.50 \\
	1000	&1999.25	&1999.26	&1995.27	&1995.24 & 0.50  \\ \hline \hline
\end{tabular}
\label{tab:table2}
\end{table}

We now extend the analysis to the same class of graphs but varying the probability of attachment and assuming to add (or remove) more than one link. Two alternative values ($p = 0.5$ and $p = 0.9$ respectively) are compared in Figures \ref{fig:fig1} and \ref{fig:fig2}. In particular, by applying the same methodology previously described, we randomly generate a graph $G$ and then we randomly add (or remove) $h$ links (with $h\geq1$). To the best of our knowledge, no bounds have been provided in the literature for $h>1$.
On one hand, we observe that a good approximation is also assured when a large number of links is added or removed. On the other hand, when the attachment probability increases, our bound provides a best proxy of the exact value since the graph $G'$ is closer to the complete graph. \\
Finally, as expected, when a lower number of vertices is considered (see Figure \ref{fig:fig2}), each link has a greater effect on the robustness of the network. Indeed, we observe  a remarkable impact on the Kirchhoff Index. Even in this case, bounds enable to capture the behaviour of the robustness measure.

\begin{figure}
	\centering
\includegraphics[height=5cm,width=5cm]{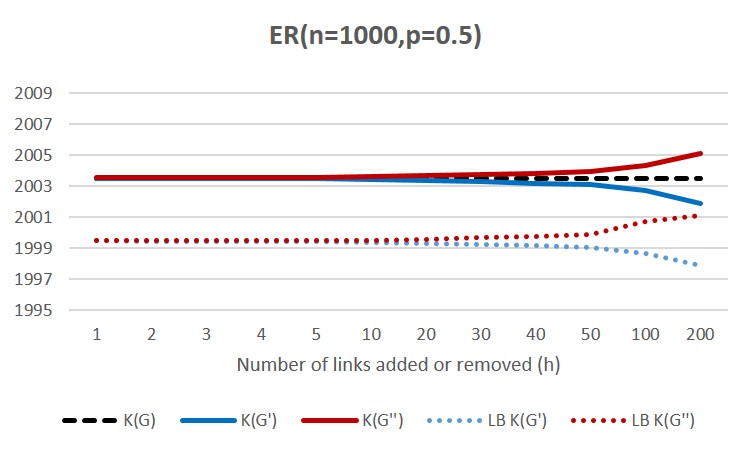}
\includegraphics[height=5cm,width=5cm]{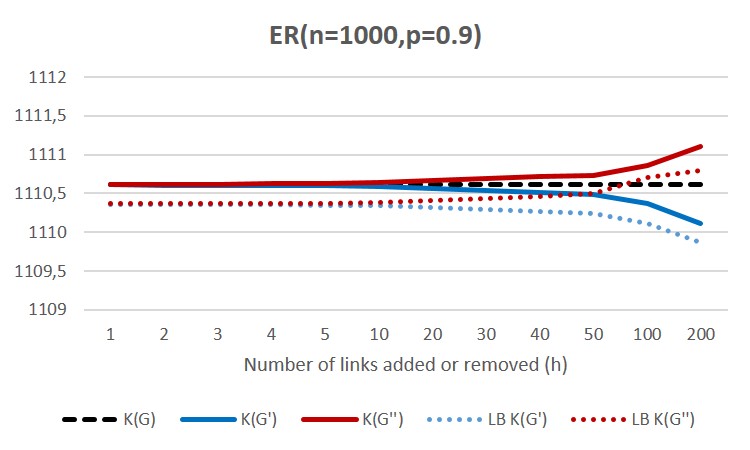}
\caption{On the left-hand side, we report the value of the Kirchhoff Index $K(G)$ (dotted black line) for a graph $G(n,m)$ randomly generated with a $G_{ER}(1000,0.5)$ model. Graphs $G'(n,m+h)$ and $G''(n,m-h)$ are obtained starting from the graph $G$ and varying the number of edges $h$ added or removed. Solid lines display the exact values of the Kirchhoff Index $K(G')$ and $K(G'')$. Dotted lines show the lower bounds (\ref{KG'})  and (\ref{KG''}) for each graph. On the right-hand side, the same values are obtained starting from a graph $G$ derived by $G_{ER}(1000,0.9)$ model.}      
\label{fig:fig1} 
\end{figure}

\begin{figure}
	\centering
\includegraphics[height=5cm,width=5cm]{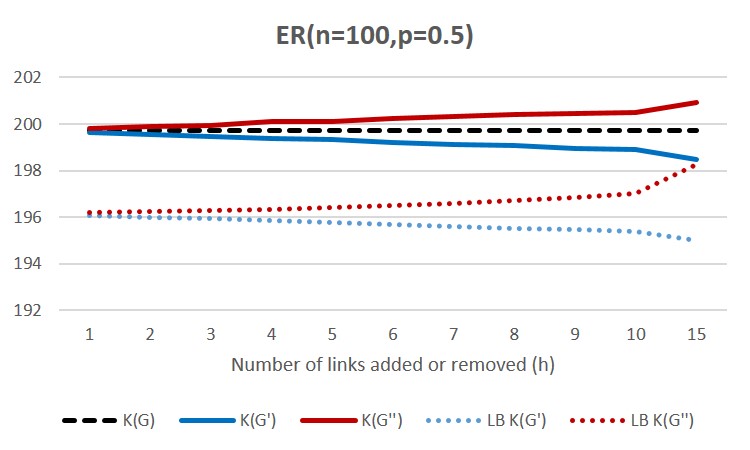}
\includegraphics[height=5cm,width=5cm]{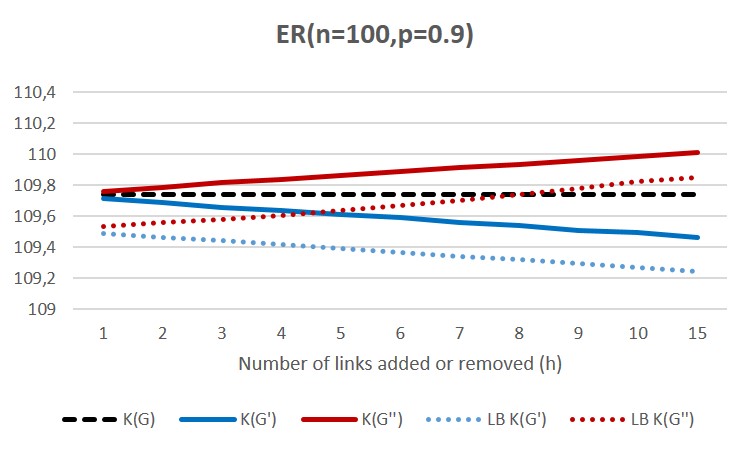}
\caption{On the left-hand side, we report the value of the Kirchhoff Index $K(G)$ (dotted black line) for a graph $G(n,m)$ randomly generated with a $G_{ER}(100,0.5)$ model. Graphs $G'(n,m+h)$ and $G''(n,m-h)$ are obtained starting from the graph $G$ and varying the number of edges $h$ added or removed. Solid lines display the exact values of the Kirchhoff Index $K(G')$ and $K(G'')$. Dotted lines show the lower bounds (\ref{KG'})  and (\ref{KG''}) for each graph. On the right-hand side, the same values are obtained starting from a graph $G$ derived by $G_{ER}(100,0.9)$ model.}     %
\label{fig:fig2}   
\end{figure}

\section{Conclusions}
By using an approach for localizing some relevant graph topological indices based on the optimization of Schur-convex or Schur-concave functions, we provide novel lower bounds on the Kirchhoff Index of graphs obtained after one link addition or removal. Furthermore, we also define new limitations when the topological change regards more than one link. \\
Analytical and numerical results show the performance of these bounds on different graphs. In particular, the bounds perform better with respect to best-known results in the literature.
Further research regards a generalization to weighted and/or directed networks and the analysis of the correlation between alternative topological metrics.


\bibliographystyle{plain}
\bibliography{biblio}

\end{document}